\documentclass[11pt]{amsart}

\usepackage{amsmath,amssymb,amsthm,mathtools}
\usepackage[hidelinks]{hyperref}
\usepackage{booktabs}

\numberwithin{equation}{section}
\newtheorem{theorem}{Theorem}[section]
\newtheorem{lemma}[theorem]{Lemma}
\newtheorem{proposition}[theorem]{Proposition}
\newtheorem{conjecture}{Conjecture}
\theoremstyle{remark}
\newtheorem{remark}[theorem]{Remark}
\newcommand{\C}{\mathbb C}
\newcommand{\N}{\mathbb N}
\newcommand{\R}{\mathbb R}
\newcommand{\Tr}{\operatorname{Tr}}
\newcommand{\rank}{\operatorname{rank}}
\newcommand{\diag}{\operatorname{diag}}
\newcommand{\eps}{\varepsilon}
\newcommand{\dd}{\,dt}

\title[A limit interchange for two-block entropy]{Interchanging limits in the mutual information of two intervals separated by one site}
\author[C. Bellavita]{Carlo Bellavita}
\address{Universit\`a degli Studi di Milano, Milano, Italy.}
\email{carlo.bellavita@unimi.it ; carlo.bellavita@gmail.com}
\author[J. A. Virtanen]{Jani A. Virtanen}
\address{University of Eastern Finland, Joensuu, Finland. University of Reading, Reading, UK. University of Helsinki, Helsinki, Finland}
\email{jani.virtanen@uef.fi}

\begin{document}
\begin{abstract}
In the paper \textit{J.~Phys.~A:~Math.~Theor.}~53 (2020), 345303, the limiting mutual information between two intervals separated by one lattice site in a free-fermion chain was evaluated. The calculation replaces the entropy function by a regularized version depending on $\varepsilon>0$, takes the block lengths to infinity, and only then lets $\varepsilon$ tend to zero. We prove that the opposite order gives the same constant $2\log2-1$. Moreover, we show that the same limit is obtained whenever both block lengths tend to infinity and $\varepsilon$ tends to zero simultaneously, with no restriction on their relative rates. In fact, the error caused by the regularization is bounded by $(1+\varepsilon)\log(1+\varepsilon)-\varepsilon\log\varepsilon$, uniformly in both block lengths.
\end{abstract}

\maketitle

\section{Introduction}

We study the limit of a determinant-based entropy difference for two intervals of lengths $m,n\in\N$ separated by one lattice site in a free-fermion chain. The question is whether the finite-block quantity introduced in \cite{BGIKMMV20} tends to $2\log2-1$ as both lengths tend to infinity.

Let $g(e^{i\theta})=\operatorname{sgn}(\cos\theta)$. The Fourier coefficients of $g$ are given by
\begin{equation}
    g_0=0,\qquad g_j=\frac{2\sin(\pi j/2)}{\pi j}\quad(j\in\mathbb Z\setminus\{0\}).
\end{equation}
Set $T_\ell[g]=(g_{j-k})_{j,k=1}^\ell$ and $u_\ell=(g_1,\ldots,g_\ell)^T$. The two matrices we compare are
\begin{equation}\label{e:matrices}
    A_{m,n}=\begin{pmatrix}
    -T_m[g]&u_mu_n^T\\
    u_nu_m^T&-T_n[g]
    \end{pmatrix},\qquad
    B_{m,n}=\begin{pmatrix}
    -T_m[g]&0\\0&-T_n[g]
    \end{pmatrix}.
\end{equation}
Thus $B_{m,n}$ is obtained from $A_{m,n}$ by setting its off-diagonal blocks to zero. 

The entropy function is
\begin{equation}
    h(t)=-\frac{1+t}{2}\log\frac{1+t}{2}
    -\frac{1-t}{2}\log\frac{1-t}{2},\qquad -1\leq t\leq1,
\end{equation}
where $0\log0=0$. It is continuous on $[-1,1]$, but its derivative has logarithmic singularities at the endpoints. For the contour calculation, it is replaced by $h_\eps(t)=e(1+\eps,t)$, where $\eps>0$ and
\begin{equation}\label{e:regularizer}
    e(x,t)=-\frac{x+t}{2}\log\frac{x+t}{2}
    -\frac{x-t}{2}\log\frac{x-t}{2},\qquad x>1.
\end{equation}
The original entropy function is recovered as $\eps\downarrow0$. Define
\begin{equation}\label{e:traces}
\begin{split}
    J_{m,n}&=\Tr h(B_{m,n})-\Tr h(A_{m,n}),\\
    J_{m,n}^{(\eps)}&=\Tr h_\eps(B_{m,n})-\Tr h_\eps(A_{m,n}).
\end{split}
\end{equation}
Then, for each fixed pair $m,n$, continuity gives
$$
    J_{m,n}=\lim_{\eps\downarrow0}J_{m,n}^{(\eps)}.
$$

The determinant calculation in \cite{BGIKMMV20} uses an elaborate Riemann--Hilbert analysis to first let the block lengths tend to infinity with $\eps$ fixed and then lets $\eps\downarrow0$, proving that
\begin{equation}\label{e:limits}
    \lim_{\eps\downarrow0}\lim_{m,n\to\infty}J_{m,n}^{(\eps)}=2\log2-1.
\end{equation}
The question about the finite-block quantities instead asks whether
\begin{equation}\label{e:limits2}
    \lim_{m,n\to\infty}J_{m,n}
    =\lim_{m,n\to\infty}\lim_{\eps\downarrow0}J_{m,n}^{(\eps)}
    =2\log2-1.
\end{equation}
The first formula alone does not justify reversing the limits. Convergence as $\eps\downarrow0$ for each fixed pair of block lengths does not say how small the approximation error remains when those lengths grow.

We prove that, for every $m,n\geq1$ and $\eps>0$,
\begin{equation}\label{e:bound-intro}
    0\leq J_{m,n}-J_{m,n}^{(\eps)}
    \leq(1+\eps)\log(1+\eps)-\eps\log\eps.
\end{equation}
The right side tends to zero as $\eps\downarrow0$ and is independent of both block lengths. Together with \eqref{e:limits}, this proves \eqref{e:limits2}. It also shows that
$$
    J_{m_j,n_j}^{(\eps_j)}\longrightarrow2\log2-1
$$
whenever $\eps_j>0$, $\eps_j\to0$ and $\min\{m_j,n_j\}\to\infty$. Thus the same limit is obtained when the approximation is removed and both block lengths tend to infinity simultaneously, with no restriction on their relative rates.

The bound is possible because we estimate the difference of the two traces rather than each trace separately. Separate estimates would introduce a factor proportional to $m+n$. In contrast, $A_{m,n}-B_{m,n}$ has rank two, with one positive and one negative eigenvalue, so the two eigenvalue counting functions differ by at most one, independently of the matrix size. A classical finite-rank trace estimate applied to $h_\eps-h$ gives the upper bound in \eqref{e:bound-intro}, while convexity of $h_\eps-h$ gives its sign. We include elementary proofs of these facts. The trace identity behind the estimate is the finite-dimensional form of Kre\u{\i}n's trace formula \cite{BY93, Krein53}; the eigenvalue counting bound follows from the min--max principle \cite{Bhatia97}. The determinant asymptotics and the evaluation of the limiting constant remain those of \cite{BGIKMMV20}; no new determinant asymptotics are needed.

We conclude Section~\ref{sec:passage} with numerical values of $J_{m,n}$ computed directly from \eqref{e:matrices}, without regularization or determinant asymptotics. These are consistent with \eqref{e:limits2} and suggest a first-order expansion in $1/m+1/n$, formulated as Conjecture~\ref{conj:rate}.

The contour-determinant method has earlier examples in the XX-chain calculation of Jin--Korepin \cite{JK04} and the XY-chain work of Its--Jin--Korepin \cite{IJK05,IJK07}. In~\cite{IJK07}, the interchange of normalized contour limits is explicitly stated as a conjecture immediately after equation~(28), on page~9. The bordered-determinant construction for disjoint subsystems was introduced by Jin--Korepin \cite{JK11} and underlies the one-site matrices considered here. Our result concerns the difference \eqref{e:traces}; the bounded-rank comparison used in its proof does not settle the single-interval limit questions in \cite{IJK07, JK04}.

All logarithms are natural, all traces are unnormalized, and $m,n\to\infty$ means $\min\{m,n\}\to\infty$.

\section{Matrix properties and the main results}
All entries here are real, and both matrices are symmetric. The matrix $A_{m,n}$ is the $k=1$ matrix of \cite[equations (2.1)--(2.3)]{BGIKMMV20} after reversing the coordinates in its first block. Indeed, equation~(2.9) there gives the cross-block entry $g_{i-m-1}g_{-j}$, which becomes $g_rg_j$ under $i=m+1-r$. Since $g_j=g_{-j}$, this reversal leaves the Toeplitz diagonal block unchanged, and the two matrices are orthogonally similar. Consequently, by \cite[equations (2.4) and (2.10)]{BGIKMMV20}, $\Tr h(A_{m,n})$ is the von Neumann entropy of the union of the two intervals in \cite[equation (2.8)]{BGIKMMV20}, while $\Tr h(B_{m,n})$ is the sum of their individual entropies. Hence $J_{m,n}$ is the mutual information between the two intervals; the residue calculation relating the contour formula to this trace expression is carried out in Section~\ref{sec:passage}.

\begin{lemma}\label{lem:contraction}
The matrices $A_{m,n}$ and $B_{m,n}$ are self-adjoint contractions, and
\begin{equation}
    \rank(A_{m,n}-B_{m,n})=2.
\end{equation}
The two nonzero eigenvalues of $A_{m,n}-B_{m,n}$ are $\|u_m\|\,\|u_n\|$ and $-\|u_m\|\,\|u_n\|$.
\end{lemma}
\begin{proof}
The matrix $T_\ell[g]$ is a compression of multiplication by the real-valued function $g$ on the unit circle, so $-I\leq T_\ell[g]\leq I$, which gives the assertion for $B_{m,n}$. Moreover,
$$
    T_{\ell+1}[g]=\begin{pmatrix}0&u_\ell^T\\u_\ell&T_\ell[g]\end{pmatrix},
$$
because $g_0=0$ and $g_j=g_{-j}$. Taking the Schur complement of the top-left entry in $I\pm T_{\ell+1}[g]\geq0$ gives
\begin{equation}\label{e:schur}
    I\pm T_\ell[g]-u_\ell u_\ell^T\geq0.
\end{equation}
Consequently, for $s\in\{-1,1\}$,
$$
\begin{aligned}
    I+sA_{m,n}&=\diag\bigl(I-sT_m[g]-u_mu_m^T,\ I-sT_n[g]-u_nu_n^T\bigr)\\
    &\quad+\begin{pmatrix}u_m\\s u_n\end{pmatrix}
    \begin{pmatrix}u_m\\s u_n\end{pmatrix}^{\!T}\geq0,
\end{aligned}
$$
the first summand being positive semidefinite by \eqref{e:schur} and the second by inspection. Hence, $-I\le A_{m,n}\le I$. Finally, put $E=A_{m,n}-B_{m,n}$. Its range is contained in the span of $(u_m,0)^T$ and $(0,u_n)^T$. Both vectors are nonzero since $g_1=2/\pi$. On the corresponding normalized basis, $E$ has the matrix
$$
    \begin{pmatrix}
    0&\|u_m\|\,\|u_n\|\\
    \|u_m\|\,\|u_n\|&0
    \end{pmatrix}.
$$
This proves the rank and eigenvalue assertions.
\end{proof}

The two nonzero eigenvalues neither vanish nor escape as the blocks grow. By Parseval's identity $\sum_{j\in\mathbb Z}g_j^2=1$, and since $g_0=0$ and $g_{-j}=g_j$, the quantity $\|u_\ell\|^2=\sum_{j=1}^\ell g_j^2$ is nondecreasing in $\ell$ with limit $1/2$. Hence
$$
    \frac{4}{\pi^2}\leq\|u_m\|\,\|u_n\|<\frac12,\qquad m,n\geq1,
$$
and the lower bound is attained at $m=n=1$.

\begin{remark}
The contraction property of the full matrix is not an automatic consequence of the corresponding property of the Toeplitz diagonal blocks because a rank-two perturbation need not preserve contractivity. Lemma~\ref{lem:contraction} ensures that the entropy traces in \eqref{e:traces} are defined and that the prescribed contour encloses both spectra. Since $B_{m,n}$ is a contraction, the denominator of the determinant ratio \eqref{e:ratio}, defined in Section~\ref{sec:passage}, has no zeros outside $[-1,1]$; the corresponding spectral inclusion for $A_{m,n}$ makes the ratio holomorphic and nonvanishing there.
\end{remark}

Let
\begin{equation}\label{e:w}
    w(\eps)=(1+\eps)\log(1+\eps)-\eps\log\eps,\qquad \eps>0.
\end{equation}
As $\eps\downarrow0$, we have $w(\eps)=\eps\log(1/\eps)+O(\eps)$.

\begin{theorem}\label{thm:uniform}
For all $m,n\geq1$ and $\eps>0$,
\begin{equation}\label{e:uniform}
    0\leq J_{m,n}-J_{m,n}^{(\eps)}\leq w(\eps).
\end{equation}
In particular, $J_{m,n}^{(\eps)}\to J_{m,n}$ as $\eps\downarrow0$ uniformly in both block lengths.
\end{theorem}

The bound \eqref{e:uniform} is proved in Section~\ref{sec:rank}; it uses no Toeplitz asymptotics. To identify the limiting value, we use two statements about the regularized quantities:
\begin{equation*}\tag{P1}\label{e:P1}
    L_\eps:=\lim_{m,n\to\infty}J_{m,n}^{(\eps)}
    \quad\text{exists for every fixed }\eps>0,
\end{equation*}
\begin{equation*}\tag{P2}\label{e:P2}
    \lim_{\eps\downarrow0}L_\eps=L,\qquad L=2\log2-1.
\end{equation*}
Together these are the content of \cite[Theorem 2.1]{BGIKMMV20}, whose iterated limit presupposes \eqref{e:P1}. Since our argument uses \eqref{e:P1} on its own, we make it explicit in Section~\ref{sec:passage}, where we also identify $J_{m,n}^{(\eps)}$ with the contour integral occurring there.

\begin{theorem}\label{thm:main}
For the matrices \eqref{e:matrices} and the trace differences \eqref{e:traces},
\begin{equation}\label{e:main1}
    \lim_{m,n\to\infty}J_{m,n}=2\log2-1,
\end{equation}
the two iterated limits agree,
\begin{equation}\label{e:main2}
    \lim_{m,n\to\infty}\lim_{\eps\downarrow0}J_{m,n}^{(\eps)}
    =\lim_{\eps\downarrow0}\lim_{m,n\to\infty}J_{m,n}^{(\eps)}
    =2\log2-1,
\end{equation}
and there is joint convergence: for every sequence $\eps_j>0$ with $\eps_j\to0$ and every pair of sequences $m_j,n_j\in\N$ with $\min(m_j,n_j)\to\infty$,
\begin{equation}\label{e:main3}
    J_{m_j,n_j}^{(\eps_j)}\longrightarrow2\log2-1.
\end{equation}
\end{theorem}

Equation \eqref{e:main1} is the assertion recorded in \cite[equation (10.1)]{BGIKMMV20}. We prove the uniform estimate next and complete the limit argument in Section~\ref{sec:passage}.

\section{A finite-rank regularization estimate}\label{sec:rank}

The proof of Theorem~\ref{thm:uniform} rests on one classical estimate: If two self-adjoint matrices differ by a perturbation of rank at most $r$, then the traces of $f$ at the two matrices differ by at most $r\int |f'|$, independently of the dimension. Two standard facts lie behind this estimate. First, for an absolutely continuous function $f$, if $N_X(t)$ denotes the number of eigenvalues of $X$ not exceeding $t$, then
$$
    \Tr\bigl(f(X)-f(Y)\bigr)=\int_{-1}^{1}f'(t)\,\xi(t;X,Y)\dd,\qquad
    \xi(t;X,Y)=N_Y(t)-N_X(t).
$$
This is the finite-dimensional case of Kre\u{\i}n's trace formula, with $\xi$ the spectral shift function of the pair \cite{Krein53,BY93}. Second, a perturbation of rank $r$ changes the eigenvalue counting function by at most $r$, by the min--max principle \cite[Chapter III]{Bhatia97}. We give the short proof because we need a refinement that is sensitive to the signs of the eigenvalues of $X-Y$, that is, only $\max\{r_+,r_-\}$ enters, rather than $r_++r_-$. Since $A_{m,n}-B_{m,n}$ has one positive and one negative eigenvalue, this refinement gives the constant $1$ in \eqref{e:uniform}.

\begin{lemma}\label{lem:rank}
Let $X,Y$ be self-adjoint matrices of the same dimension, with spectra in $[-1,1]$, and suppose that $\rank(X-Y)\leq r$. For every absolutely continuous $f:[-1,1]\to\R$,
\begin{equation}\label{e:rankbound}
    |\Tr f(X)-\Tr f(Y)|\leq r\int_{-1}^1|f'(t)|\dd.
\end{equation}
More precisely, if $r_+$ and $r_-$ are the numbers of positive and negative eigenvalues of $X-Y$, respectively, then $r$ in \eqref{e:rankbound} may be replaced by $\max\{r_+,r_-\}$.
\end{lemma}
\begin{proof}
Put $E=X-Y$. The spectral subspace on which $E$ is nonpositive has codimension $r_+$. Its intersection with the spectral subspace of $Y$ for $(-\infty,t]$ therefore has dimension at least $N_Y(t)-r_+$. On this intersection,
$$
    \langle Xv,v\rangle
    =\langle Yv,v\rangle+\langle Ev,v\rangle\leq t\|v\|^2.
$$
The min--max principle gives $N_X(t)\geq N_Y(t)-r_+$. Applying the same argument to $Y-X$ gives $N_Y(t)\geq N_X(t)-r_-$. Consequently,
\begin{equation}\label{e:counts}
    -r_+\leq N_X(t)-N_Y(t)\leq r_-,\qquad t\in\R.
\end{equation}
In particular, $|N_X(t)-N_Y(t)|\leq\max\{r_+,r_-\}\leq r$.

For each $x\in[-1,1]$, absolute continuity gives
$$
    f(x)=f(1)-\int_{-1}^1 f'(t)\,\mathbf 1_{[x,1]}(t)\dd.
$$
Summing over the eigenvalues of $X$ and $Y$, the terms involving $f(1)$ cancel because the dimensions agree. Thus,
\begin{equation}\label{e:counttrace}
    \Tr f(X)-\Tr f(Y)=-\int_{-1}^1 f'(t)\bigl(N_X(t)-N_Y(t)\bigr)\dd.
\end{equation}
Combining \eqref{e:counttrace} and \eqref{e:counts} proves both assertions.
\end{proof}

\begin{proposition}\label{prop:regularization}
Under the hypotheses of Lemma~\ref{lem:rank},
\begin{equation}\label{e:regularizerprop}
    \left|\bigl(\Tr h_\eps(X)-\Tr h_\eps(Y)\bigr)-\bigl(\Tr h(X)-\Tr h(Y)\bigr)\right|
    \leq r\,w(\eps),
\end{equation}
with $w$ as in \eqref{e:w}. The factor $r$ may be replaced by $\max\{r_+,r_-\}$ as in Lemma~\ref{lem:rank}. In particular, the error is $O(\eps\log(1/\eps))$ as $\eps\downarrow0$, uniformly over all such pairs $X,Y$ for fixed $r$.
\end{proposition}
\begin{proof}
The entropy function $h$ is absolutely continuous on $[-1,1]$ because
$$
    h'(t)=\frac12\log\frac{1-t}{1+t},
$$
on $(-1,1)$, and its logarithmic endpoint singularities are integrable. Since $h_\eps$ is smooth on $[-1,1]$ for $\eps>0$, the difference $a_\eps=h_\eps-h$ is also absolutely continuous. Differentiating \eqref{e:regularizer} gives
\begin{equation}
    a_\eps'(t)=\frac12\log\frac{(1+\eps-t)(1+t)}{(1+\eps+t)(1-t)},\qquad -1<t<1.
\end{equation}
This derivative is odd, and it is nonnegative for $0<t<1$ because the numerator exceeds the denominator by $2\eps t$. It follows that
$$
\begin{aligned}
    \int_{-1}^1|a_\eps'(t)|\dd
    &=\int_0^1\left[\log\left(1+\frac{\eps}{1-t}\right)
    -\log\left(1+\frac{\eps}{1+t}\right)\right]\dd\\
    &\leq\int_0^1\log\left(1+\frac{\eps}{s}\right)\,ds
    =(1+\eps)\log(1+\eps)-\eps\log\eps=w(\eps).
\end{aligned}
$$
Applying Lemma~\ref{lem:rank} to $a_\eps$ gives \eqref{e:regularizerprop}, and the final assertion follows from $w(\eps)=\eps\log(1/\eps)+O(\eps)$ as $\eps\downarrow0$.
\end{proof}

\begin{remark}
The exact derivative integral is
\begin{equation}
\begin{split}
    v(\eps):=\int_{-1}^1|a_\eps'(t)|\dd
    &=2(1+\eps)\log(1+\eps)-\eps\log\eps\\
    &\quad-(2+\eps)\log(2+\eps)+2\log2.
\end{split}
\end{equation}
It satisfies $0\leq v(\eps)\leq w(\eps)$ and $w(\eps)-v(\eps)=O(\eps)$ as $\eps\downarrow0$. We can therefore replace $w(\eps)$ by $v(\eps)$ in Proposition~\ref{prop:regularization} and Theorem~\ref{thm:uniform}. We use the simpler expression $w$ in the limit argument; no optimality claim is made for either bound.
\end{remark}

\begin{proof}[Proof of Theorem~\ref{thm:uniform}]
By Lemma~\ref{lem:contraction}, the perturbation $A_{m,n}-B_{m,n}$ has exactly one positive and one negative eigenvalue. The refined part of Proposition~\ref{prop:regularization} therefore gives
$$
    |J_{m,n}^{(\eps)}-J_{m,n}|\leq w(\eps).
$$
To prove the sign, recall $a_\eps=h_\eps-h$. A second differentiation gives
$$
    a_\eps''(t)=\frac{\eps(1+\eps+t^2)}
    {(1-t^2)((1+\eps)^2-t^2)}>0,\qquad -1<t<1.
$$
Since $a_\eps$ is continuous at the endpoints, it is convex on $[-1,1]$.

We now use the convexity of $a_\eps$ to prove that
\begin{equation}\label{e:convexity}
    \Tr a_\eps(B_{m,n})\leq\Tr a_\eps(A_{m,n}).
\end{equation}
Indeed, choose an orthonormal eigenbasis $v_1,\ldots,v_{m+n}$ of $B_{m,n}$ such that each vector is supported in one of the two blocks, and let $\mu_j$ be the corresponding eigenvalues. The off-diagonal blocks do not contribute to the diagonal matrix elements of such vectors, so
$$
    \langle A_{m,n}v_j,v_j\rangle=\langle B_{m,n}v_j,v_j\rangle=\mu_j.
$$
Let $\nu_j$ be the spectral measure of $A_{m,n}$ at $v_j$. Since $\|v_j\|=1$, $\nu$ is a probability measure, and it is supported in $[-1,1]$ by Lemma~\ref{lem:contraction}. Its barycenter is $\mu_j$ by the previous display,
$$
    \int_{-1}^1 t\,d\nu_j(t)=\langle A_{m,n}v_j,v_j\rangle=\mu_j,
$$
so Jensen's inequality gives
$$
    \int_{-1}^1 a_\eps(t)\,d\nu_j(t)\geq a_\eps(\mu_j).
$$
Summing over $j$, and using
$$
    \Tr a_\eps(A_{m,n})=\sum_j\int_{-1}^1 a_\eps(t)\,d\nu_j(t),
    \qquad\Tr a_\eps(B_{m,n})=\sum_j a_\eps(\mu_j),
$$
yields \eqref{e:convexity}. Therefore,
$$
    J_{m,n}-J_{m,n}^{(\eps)}=\Tr a_\eps(A_{m,n})-\Tr a_\eps(B_{m,n})\geq0,
$$
which proves \eqref{e:uniform}.
\end{proof}

\section{The contour formula and the limit passage}\label{sec:passage}
We first show that $J_{m,n}^{(\eps)}$ is represented by the determinant ratio used in \cite{BGIKMMV20}. Set
\begin{equation}\label{e:ratio}
    \Delta_{m,n}(\lambda)
    =\frac{\det(\lambda I-A_{m,n})}{\det(\lambda I-B_{m,n})}
    =\frac{\det(\lambda I-A_{m,n})}{D_m[\lambda+g]D_n[\lambda+g]},
\end{equation}
where $D_\ell[\phi]=\det T_\ell[\phi]$. This is the ratio denoted by $\widehat D$ in \cite[Theorem 2.1]{BGIKMMV20}. By Lemma~\ref{lem:contraction}, it is holomorphic and nonzero on $\C\setminus[-1,1]$.

Let $\Gamma_\eps$ be the circle $|\lambda|=1+\eps/2$, oriented counterclockwise, as allowed in \cite[Section 2, following equation (2.4)]{BGIKMMV20}. Use the principal logarithms in $e(1+\eps,\lambda)$; since the branch points lie at $\pm(1+\eps)$, this function is holomorphic on a neighborhood of the closed disk bounded by $\Gamma_\eps$. The residue theorem gives
\begin{equation}\label{e:contour}
    J_{m,n}^{(\eps)}=-\frac1{2\pi i}\int_{\Gamma_\eps}
    e(1+\eps,\lambda)\frac{\Delta_{m,n}'(\lambda)}{\Delta_{m,n}(\lambda)}\,d\lambda.
\end{equation}
Indeed, the logarithmic derivative of each characteristic determinant is the sum of $(\lambda-\mu)^{-1}$ over its eigenvalues $\mu$, counted with multiplicity, and all of these lie inside $\Gamma_\eps$ by Lemma~\ref{lem:contraction}. Formula~\eqref{e:contour} therefore agrees exactly with the trace definition \eqref{e:traces}. Writing the logarithmic derivative as a quotient avoids any choice of a branch of $\log\Delta_{m,n}$.

We next make \eqref{e:P1} explicit by combining results of \cite{BGIKMMV20}. The determinant reduction in \cite[Section 3]{BGIKMMV20}, together with Proposition~3.1 there, gives
\begin{equation}
    \Delta_{m,n}(\lambda)=1-q_m(\lambda)q_n(\lambda),\qquad
    q_\ell(\lambda)=u_\ell^T(\lambda I+T_\ell[g])^{-1}u_\ell,
\end{equation}
where the pairing is bilinear rather than sesquilinear; each $q_\ell$ is holomorphic on $\C\setminus[-1,1]$. Lemma~3.2 of that paper establishes
\begin{equation}\label{e:qresult}
    q_\ell(\lambda)=i\tan\left(\frac\pi2\beta(\lambda)\right)+O(\ell^{-1/4}),
    \qquad \beta(\lambda)=\frac1{2\pi i}\log\frac{\lambda+1}{\lambda-1},
\end{equation}
uniformly on compact subsets of $|\lambda|>1$, where the logarithm is principal. Hence, $\Delta_{m,n}\to\Delta_\infty:=1+\tan^2\bigl(\tfrac\pi2\beta\bigr)$ locally uniformly there, and $\Delta_\infty$ is nonvanishing by \cite[Lemma 9.1]{BGIKMMV20}. Fix $\eps>0$. Then $|\Delta_\infty|\geq c>0$ on a compact neighborhood of $\Gamma_\eps$, so $|\Delta_{m,n}|\geq c/2$ there for all large $m,n$; local uniform convergence also gives convergence of the derivatives on that neighborhood, by the Cauchy estimates. Therefore $\Delta_{m,n}'/\Delta_{m,n}\to\Delta_\infty'/\Delta_\infty$ uniformly on $\Gamma_\eps$, and \eqref{e:contour} yields \eqref{e:P1} with
$$
    L_\eps=-\frac1{2\pi i}\int_{\Gamma_\eps}e(1+\eps,\lambda)
    \frac{\Delta_\infty'(\lambda)}{\Delta_\infty(\lambda)}\,d\lambda.
$$
This is the quantity whose $\eps\downarrow0$ limit is computed in \cite[Section 9]{BGIKMMV20}, giving \eqref{e:P2}. We emphasize that no uniformity of \eqref{e:qresult} as $\eps\downarrow0$ is assumed anywhere below.

\begin{proof}[Proof of Theorem~\ref{thm:main}]
Fix $\eps>0$. By the triangle inequality,
$$
    |J_{m,n}-L|\leq |J_{m,n}-J_{m,n}^{(\eps)}|
    +|J_{m,n}^{(\eps)}-L_\eps|+|L_\eps-L|.
$$
Taking the upper limit as $m,n\to\infty$ and using Theorem~\ref{thm:uniform} and \eqref{e:P1}, we find
$$
    \limsup_{m,n\to\infty}|J_{m,n}-L|\leq w(\eps)+|L_\eps-L|.
$$
The left side does not depend on $\eps$; letting $\eps\downarrow0$ and using \eqref{e:P2} makes both terms on the right tend to zero, which proves \eqref{e:main1}. At each fixed $m,n$ the inner limit as $\eps\downarrow0$ equals $J_{m,n}$, so \eqref{e:main1} gives the first equality in \eqref{e:main2}, and \eqref{e:P1}--\eqref{e:P2} give the second.

Finally, for the sequences in \eqref{e:main3},
$$
    |J_{m_j,n_j}^{(\eps_j)}-L|
    \leq w(\eps_j)+|J_{m_j,n_j}-L|\to 0,
$$
which is joint convergence with no restriction on the relative growth of the two block lengths.
\end{proof}

\begin{remark}
The comparison used here is available because $A_{m,n}$ and $B_{m,n}$ differ by a perturbation of rank two. It does not by itself settle the single-interval limit passages in \cite{JK04,IJK07}: The endpoint normalization in \cite[equations (27)--(28)]{IJK07} is a subtraction rather than a bounded-rank comparison, so a different uniform estimate is needed there.
\end{remark}

\begin{remark}
No convergence rate in the block lengths is proved here. Table~\ref{tab:numerics} records double-precision values of $J_{m,n}$ computed directly from the eigenvalues of \eqref{e:matrices}. For equal block lengths we use the exact orthogonal decomposition into $-T_m[g]+u_mu_m^T$ and $-T_m[g]-u_mu_m^T$. The even Fourier coefficients are set exactly to zero; eigenvalues lying outside $[-1,1]$ only by round-off are clipped to the interval before evaluating $h$. The table includes odd and mixed-parity lengths as well as unequal block lengths. 
\end{remark}

\begin{table}[htbp]
\centering
\begin{tabular}{rrrr}
\toprule
$m$ & $n$ & $J_{m,n}$ & $(J_{m,n}-L)/(1/m+1/n)$\\
\midrule
$64$ & $64$ & $0.381140377$ & $-0.164928$\\
$127$ & $127$ & $0.383683587$ & $-0.165784$\\
$128$ & $128$ & $0.383703877$ & $-0.165791$\\
$256$ & $256$ & $0.384995710$ & $-0.166227$\\
$512$ & $512$ & $0.385644179$ & $-0.166447$\\
$1024$ & $1024$ & $0.385969055$ & $-0.166557$\\
$1500$ & $1500$ & $0.386072239$ & $-0.166591$\\
\midrule
$127$ & $128$ & $0.383693731$ & $-0.165788$\\
$511$ & $512$ & $0.385643544$ & $-0.166446$\\
\midrule
$64$ & $1024$ & $0.383544962$ & $-0.165611$\\
$200$ & $1000$ & $0.385296548$ & $-0.166302$\\
\bottomrule
\end{tabular}
\caption{Numerical one-site entropy differences, with $L=2\log2-1$. The entries are floating-point computations, not certified error bounds.}
\label{tab:numerics}
\end{table}

The ratios in the last column are consistent with convergence to $-1/6$, suggesting the following first-order expansion.

\begin{conjecture}\label{conj:rate}
As $\min\{m,n\}\to\infty$, the trace differences \eqref{e:traces} satisfy
$$
    J_{m,n}=2\log2-1-\frac16\left(\frac1m+\frac1n\right)
    +o\left(\frac1m+\frac1n\right).
$$
\end{conjecture}

The estimate \eqref{e:qresult} concerns an inverse-Toeplitz quantity on fixed compact subsets of $|\lambda|>1$. It does not by itself control the error after letting $\eps\downarrow0$ and therefore does not establish Conjecture~\ref{conj:rate}.

\section*{Acknowledgments}
Bellavita is member of the Gruppo Nazionale per l’Analisi Matematica, la Probabilità e le loro Applicazioni (GNAMPA) of the Istituto Nazionale di Alta Matematica (INdAM). Virtanen was supported in part by the Engineering and Physical Sciences Research Council grant EP/Y008375/1 and the Research Council of Finland grant 375631.

\end{document}